\documentclass[a4paper]{article}
\usepackage{amsmath}
\usepackage{amsfonts}
\usepackage{amssymb}
\usepackage{amsthm}
\usepackage{mathtools}
\usepackage{tikz}
\usepackage{tikz-cd}
\usepackage{enumerate}
\usetikzlibrary{patterns}
\DeclareMathOperator{\id}{id}

\DeclareMathOperator{\rk}{rk}
\newcommand{\bdy}{\ensuremath{\partial}}

\newcommand{\iso}{\ensuremath{\cong}}
\newcommand{\Z}[1][]{\ensuremath{\mathbb{Z}_{#1}}}
\newcommand{\Q}{\ensuremath{\mathbb{Q}}}

\newcommand{\N}{\ensuremath{\mathbb{N}}}

\newcommand{\F}{\ensuremath{\mathbb{F}}}

\newcommand{\nsgp}[1][]{\ensuremath{\triangleleft_{#1}}}

\newcommand{\gp}[1]{\ensuremath{\langle #1\rangle}}

\newcommand{\pRm}[2][R]{\ensuremath{#1[\![#2]\!]}}
\newtheorem{theorem}{Theorem}[section]

\newtheorem{prop}[theorem]{Proposition}

\theoremstyle{definition}
\newtheorem{defn}[theorem]{Definition}

\theoremstyle{remark}
\newtheorem*{rmk}{Remark}
\newtheorem{example}[theorem]{Example}
\theoremstyle{plain}
\newtheorem*{thmquoteA}{Theorem \ref{thmUBEG}}
\newtheorem*{thmquoteB}{Theorem \ref{thmInacc}}
\newtheorem*{thmquoteC}{Theorem \ref{ThmRespInacc}}
\newcounter{introthmcount}
\theoremstyle{definition}

\title{On accessibility for pro-$p$ groups}
\author{Gareth Wilkes}
\begin{document}
\maketitle
\begin{abstract}
We define the notion of accessibility for a pro-$p$ group. We prove that finitely generated pro-$p$ groups are accessible given a bound on the size of their finite subgroups. We then construct a finitely generated inaccessible pro-$p$ group, and also a finitely generated inaccessible discrete group which is residually $p$-finite.
\end{abstract}
\section{Introduction}
A common thread in mathematics is the search for conditions which tame wild objects in some way and allow deeper analysis of their structure. In the case of groups finiteness properties, finite dimension properties, and accessibility all fall into this thread. Accessibility is the name given to the following concept. Given a group, it may be possible to split the group over a finite subgroup into pieces which could be expected to be somehow `simpler'. For a notion of `simpler' to be of much help, this process should terminate---it should not be possible to continue splitting the group {\it ad infinitum}.  Once we have split as far as possible, the remaining pieces are one-ended and amenable to further study---in this way accessibility is the foundation of JSJ theory for groups.

We adopt the following definition of accessibility, which is well known to be equivalent to the original classical definition of Wall \cite{Wall71, SW79} in the case of (finitely generated) discrete groups.
\begin{defn}
Let $G$ be a discrete group. We say $G$ is {\em accessible} if there is a number $n=n(G)$ such that any finite, reduced graph of discrete groups $\cal G$ with finite edge groups having fundamental group isomorphic to $G$ has at most $|EX|\leq n$ edges.
\end{defn}
That is, `one cannot continually split $G$ over finite subgroups'. Examples of accessible discrete groups include the finitely generated torsion free groups (by Grushko's theorem), finitely generated groups with a bound on the size of their finite subgroups \cite{Linnell83}, and all finitely presented groups \cite{Dunwoody85}. However, contrary to a famous conjecture of Wall \cite{Wall71}, not all finitely generated groups are accessible \cite{Dunwoody93}.

Pro-$p$ groups are a class of compact topological groups with many pleasant features. See \cite{DdSMS03} for an introduction. Being `built out of' finite $p$-groups, the many special properties of $p$-groups may be used to analyse pro-$p$ groups. At the same time the category of pro-$p$ groups has surprising similarities to the category of discrete groups. There are close relations between the cohomology theories of the two classes \cite{SerreCohom}, and pro-$p$ groups have a good theory of actions on trees whose results parallel the Bass-Serre theory for discrete groups \cite{RZ00p, Ribes17}. In this way one can often prove results for pro-$p$ groups which parallel those for discrete groups, even when the discrete groups are studied geometrically so that the techniques of the proof do not necessarily translate over. Examples include, for instance, \cite{labute67, Lubotzky82, Wilkes17c}.

In this paper we will commence the study of accessibility for pro-$p$ groups. In Section \ref{secProperness} we will give formal definitions of graphs of pro-$p$ groups and their fundamental groups. We also study the concept of properness, a vital extra component in this theory not present in the discrete world. We then derive certain properties of these notions which are necessary and expedient for the rest of the paper. 

In Section \ref{secUBEG} we will prove a pro-$p$ version of Linnell's theorem \cite{Linnell83}.
\begin{thmquoteA}
Let $G$ be a finitely generated pro-$p$ group and let $K$ be an integer. Let $(X, G_\bullet)$ be a proper reduced finite graph of groups decomposition of $G$ such that each edge group $G_e$ has size at most $K$. Then $X$ has at most
\[\frac{pK}{p-1}(\rk(G)-1) + 1 \]
edges.
\end{thmquoteA}
In Section \ref{secInacc} we reach the primary purpose of the paper: the explicit construction of an inaccessible finitely generated pro-$p$ group.
\begin{thmquoteB}
There exists a finitely generated inaccessible pro-$p$ group $J$.
\end{thmquoteB}
Finally in Section \ref{secRespInacc} we note that our construction gives an inaccessible discrete group which is residually $p$-finite.
\begin{thmquoteC}
There exists a finitely generated, inaccessible, residually $p$-finite discrete group.
\end{thmquoteC}

\section{Properness of graphs of pro-$p$ groups}\label{secProperness}
In this section we will recall various notions concerning graphs of pro-$p$ groups, and establish some properties of them which do not seem to appear in the standard literature. In this paper an (oriented) graph $X$ will consist of a set $VX$ of vertices, a set $EX$ of edges, and two functions $d_k\colon EX\to VX$ for $k\in\{0,1\}$ describing the endpoints of $e$. We will generally ignore the orientation when describing a graph, but it is occasionally useful when greater precision is necessary.
\begin{defn}
A graph ${\cal G}=(X,G_\bullet)$ of pro-$p$ groups consists of a finite graph $X$, a pro-$p$ group $G_x$ for each $x\in X$, and monomorphisms $\bdy_{e,k}\colon G_e\to G_{d_k(e)}$ for each edge $e\in EX$ and for $k=0,1$. A graph of groups is {\em reduced} if $\bdy_{e,k}$ is never an isomorphism when $e$ is not a loop.
\end{defn}
\begin{rmk}
One may define graphs of groups over arbitrary profinite graphs $X$ (see \cite[Chapter 6]{Ribes17}), but this adds an extra level of complication and subtlety which is not required in the present paper.
\end{rmk}
\begin{defn}\label{FGDefn}
The full definition of a fundamental group of an arbitrary graph of pro-$p$ groups may be found in \cite[Section 6.2]{Ribes17}. We give here a simplified version \cite[Example 6.2.3(c)]{Ribes17} appropriate for the study of finite graphs of pro-$p$ groups. 

Let ${\cal G}=(X,G_\bullet)$ be a finite graph of pro-$p$ groups and choose a maximal subtree $T$ of $X$. A {\em specialisation} of $\cal G$ is a triple $(H,\nu_\bullet, t_\bullet)$ consisting of a pro-$p$ group $H$, morphisms $\nu_x\colon G_x\to H$ for each $x\in X$, and elements $t_e\in H$ for $e\in EX$, such that:
\begin{itemize}
\item $t_e=1$ for all $e\in ET$
\item $\nu_{d_0(e)}(\bdy_{e,0}(g)) = \nu_e(g) = t_e\nu_{d_1(e)}(\bdy_{e,1}(g))t_e^{-1} $ for all $e\in EX, g\in G_e$.
\end{itemize}

A {\em pro-$p$ fundamental group} of $\cal G$ is a `universal' specialisation---that is, a specialisation $(H,\nu_\bullet, t_\bullet)$ of $\cal G$ such that for any other specialisation $(K,\beta_\bullet, t'_\bullet)$ of $\cal G$ there exists a unique morphism $\phi\colon H\to K$ such that $\phi\nu_x=\beta_x$ for all $x\in X$ and $\phi(t_e)=t'_e$ for all $e\in E$. Note that the $\nu_x$ are {\em not} required to be injections.

We denote the pro-$p$ fundamental group by $\Pi_1({\cal G})$ or $\Pi_1(X,G_\bullet)$. We also use the notation $G_x\amalg_{G_e} G_y$ for the fundamental group of a graph of groups with two vertices $x$ and $y$ and one edge $e$.
\end{defn}
An equivalent, but slightly less formal, definition would be to define $\Pi_1(\cal G)$ to be the pro-$p$ group given by the pro-$p$ presentation
\begin{multline*}
\big<G_x\, (x\in X),\, t_e\, (e\in EX) \, \big| \, t_e=1 \text{ (for all }e\in ET)\\ \bdy_{e,0}(g) = g = t_e\bdy_{e,1}(g)t_e^{-1} \text{ (for all } e\in EX, g\in G_e)\big> 
\end{multline*}
The fundamental pro-$p$ group of a graph of groups always exists \cite[Proposition 6.2.1(b)]{Ribes17} and is independent of the choice of $T$ \cite[Theorem 6.2.4]{Ribes17}.
\begin{rmk}
One can also consider a graph of pro-$p$ groups as a graph of profinite groups, and consider the above definition of `fundamental group' in the category of all profinite groups. The profinite fundamental group is not in general isomorphic to the pro-$p$ fundamental group, and the two concepts should be carefully distinguished if both are present. However in this paper we only consider pro-$p$  fundamental groups, so we will not clutter the notation by distinguishing it from the profinite fundamental group.
\end{rmk}

Let ${\cal G}=(X,G_\bullet)$ be a finite graph of {\em finite} $p$-groups. We may consider $\cal G$ as a graph of pro-$p$ groups as above, or we may consider it as a graph of finite discrete groups and consider the fundamental group of $\cal G$ in the category of discrete groups \cite[Section 5.1]{SerreTrees}. We denote the discrete fundamental group by $\pi_1(\cal G)$ or $\pi_1(X,G_\bullet)$. The pro-$p$ presentation above makes clear the following relation between the discrete and pro-$p$ fundamental groups.
\begin{prop}[Proposition 6.5.1 of \cite{Ribes17}]\label{ProPvsDisc}
Let ${\cal G}= (X,G_\bullet)$ be a finite graph of finite $p$-groups. Then $\Pi_1(\cal G)$ is naturally isomorphic to the pro-$p$ completion of $\pi_1(\cal G)$.
\end{prop}

In Definition \ref{FGDefn} we noted that the natural maps $G_x\to \Pi_1(X,G_\bullet)$ are not required to be injections. This is in sharp contrast to the situation for graphs of discrete groups, where injectivity is automatic. However graphs of groups are most useful when this injectivity does hold, so we give this property a name.
\begin{defn}
A graph of pro-$p$ groups ${\cal G}=(X,G_\bullet)$ is {\em proper} if the natural maps $G_x\to \Pi_1({\cal G})$ are injections for all $x\in X$.
\end{defn}
A proper, reduced graph of pro-$p$ groups whose fundamental pro-$p$ group is isomorphic to a pro-$p$ group $G$ will be referred to as a {\em splitting} of $G$.
\begin{rmk}
There is a bifurcation of terminology in the literature. One finds discussions of `injective' graphs of groups \cite{Ribes17} and of `proper' amalgamated free products and HNN extensions \cite{RZ00} (which are of course the  basic examples of graphs of groups). There are also papers with the convention that only proper amalgamated free products `exist' \cite{Ribes71}. The preference of the present author is to use the word `proper' uniformly.  
\end{rmk}
For a graph of finite $p$-groups ${\cal G}=(X,G_\bullet)$, properness is equivalent both to the existence of a finite quotient of $\Pi_1({\cal G})$ into which all $G_x$ inject and to the property that the {\em discrete} fundamental group $\pi_1({\cal G})$ (that is, the fundamental group when considered as a graph of discrete groups) is residually $p$. There are classical criteria for this property in the case of one-edge graphs of groups \cite{Hig64, chat94}. Criteria for more general graphs of groups also exist \cite{AF13, Wilkes18e}.
\begin{defn}
Let $G$ be a pro-$p$ group. We say $G$ is {\em accessible} if there is a number $n=n(G)$ such that any finite, proper, reduced graph of pro-$p$ groups $\cal G$ with finite edge groups having fundamental group isomorphic to $G$ has at most $|EX|\leq n$ edges.
\end{defn}
Generally speaking, proper graphs of groups are the only ones worth considering. An illustration of this principle in the context of accessibility will be given later (Example \ref{ArbLargeImprop}).

We define an `inverse system of graphs of groups' in the following way. Fix a graph $X$ and an inverse system $(I,\prec)$. An inverse system of graphs of groups $({\cal G}_i)_{i\in I}$ is a family of graphs of groups ${\cal G}_i = (X, G_{i,\bullet})$ together with surjective transition maps $G_{i,x}\to G_{j,x}$ for all $x\in X$ and $i\succ j$, such that the diagrams 
\[ \begin{tikzcd}
G_{i,e}\ar[hook]{r} \ar{d} &  G_{i,d_k(x)}\ar{d}\\
G_{j,e}\ar[hook]{r} &  G_{j,d_k(x)}
\end{tikzcd}\]
commute for all $e\in EX$, $k=0,1$ and all $i\succ j$. There is then a natural {\em inverse limit graph of groups} ${\cal H}=(X, H_\bullet)$ where $H_x=\varprojlim_i G_{i,x}$ for each $x\in X$.
\begin{prop}\label{invlimgofgs}
Take an inverse system of graphs of groups ${\cal G}_i = (X, G_{i,\bullet})$ with inverse limit graph of groups ${\cal H}=(X, H_\bullet)$.
\begin{enumerate}
\item Taking fundamental groups commutes with inverse limits---that is, we have $\varprojlim \Pi_1({\cal G}_i) = \Pi_1({\cal H})$.
\item If each ${\cal G}_i$ is proper then ${\cal H}$ is proper.
\end{enumerate}
\end{prop}
\begin{proof}
By the universal property defining $\Pi_1(\cal H)$, for each $i$ there is a natural surjection $\Psi_i\colon \Pi_1({\cal H})\to \Pi_1({\cal G}_i)$. Hence we have a natural surjection $\Psi\colon \Pi_1({\cal H})\to \varprojlim\Pi_1({\cal G}_i)$. 

To verify that $\Psi$ is an isomorphism, take $h\in \Pi_1({\cal H})\smallsetminus \{1\}$. There is some finite $p$-group $K$ and a map $\phi\colon \Pi_1({\cal H})\to K$ such that $\phi(h)\neq 1$. For each $x\in X$, the natural map 
\[ \varprojlim G_{i,x} \to \Pi_1({\cal H})\to K\]
is continuous, and hence factors through some $G_{i_x,x}$. Since $X$ is a finite graph, we may choose $j\in I$ such that $j\succ i_x$ for all $x$. The natural maps $G_{j,x}\to K$ are compatible with the monomorphisms in the definition of the graph of groups ${\cal G}_i$, and hence give rise to a natural map $\Pi_1({\cal G}_i)\to K$ through which $\phi$ factors. It follows that $\Psi_i(h)\neq 1$, so that $\Psi(h)\neq 1$ and we have proved part 1 of the theorem.
\[\begin{tikzcd}
\varprojlim G_{i,x} \ar{r}\ar{d} & \Pi_1({\cal H}) \ar{r}{\Psi_i} \ar{dd}{\phi}& \Pi_1({\cal G}_j)\ar[dotted]{ddl}{\exists}\\
G_{j,x} \ar{dr} & & \\
& K &
\end{tikzcd}\]
The second part of the theorem holds because the map $\varprojlim G_{i,x} \to \Pi_1({\cal H})$ is the inverse limit of the maps $G_{j,x}\to \Pi_1({\cal G}_j)$ and because inverse limits of injective maps are injective.
\end{proof}
When working with graphs of groups one very often conducts proofs in an inductive fashion---for instance by proving a statement for a subgraph of groups with one fewer edge and then adding in the remaining edge. We will therefore include the following warning: properness is a property of an {\em entire} graph of groups, and not a property which may be studied edge-by-edge. We give the following example. The graph is simply a line segment with three vertices and two edges. The amalgamation over either edge is individually proper, but the full graph of groups is improper.
\begin{example}
Take four copies $G_1,\ldots, G_4$ of the mod-$p$ Heisenberg group 
\[\big< x,y  \,\big|\, x^p=y^p=[x,y]^p=1, [x,y]\text{ central } \big> \]
whose given generating sets shall be denoted $\{x_1,y_1\},\ldots, \{x_4, y_4\}$. Furthermore take two copies $H_1$ and $H_2$ of $\F_p^2$, with generating sets $\{u_1, v_1\}$ and $\{u_2, v_2\}$ respectively. Define a graph of groups
\begin{equation}\label{BadThreeEdge}
\begin{tikzcd}
G_1 \ar[dash]{r}{H_1}& G_2 \times G_3 \ar[dash]{r}{H_2} & G_4
\end{tikzcd}
\end{equation}
with inclusions given by
\[\begin{tikzcd}[row sep=tiny]
x_1&\ar[mapsto]{l} u_1 \ar[mapsto]{r} & {[x_2, y_2]} & & x_2 & \ar[mapsto]{l} u_2 \ar[mapsto]{r}& {[x_4, y_4]}\\
{[x_1, y_1]} &\ar[mapsto]{l} v_1 \ar[mapsto]{r} & x_3 && {[x_3,y_3]} &\ar[mapsto]{l} v_2 \ar[mapsto]{r} & x_4
\end{tikzcd}\]
This graph of groups is improper: the identifications $x_1=[x_2,y_2]$ et cetera in its fundamental group imply a relation
\[x_1 = [[[[x_1,y_1],y_3],y_4],y_2] \]
so that $x_1$ may be expressed as a commutator of arbitrary length, and therefore must vanish in a pro-$p$ group. So $G_1$ cannot inject into the fundamental pro-$p$ group of the graph of groups \eqref{BadThreeEdge}. However either amalgamation over just one edge is proper, as may be readily verified using Higman's criterion \cite{Hig64}.
\end{example}
There is thus a warning attached to working with sub-graphs of groups in some cases. However once properness of the whole graph of groups is established one may indeed manipulate subgraphs in the expected manner. The next proposition effectively allows us to `bracket' sections of a proper graph of groups for individual study.

Let $X$ be a connected finite graph and let ${\cal G} = (X, G_\bullet)$ be a proper graph of pro-$p$ groups with fundamental group $G=\Pi_1({\cal G})$. Let $Y$ be a connected subgraph of $X$. Let $Z=X/Y$ be the quotient graph, where we denote the image of a point $x\in X\smallsetminus Y$ by $[x]$ and denote the image in $Z$ of $Y$ by $[Y]$. Let ${\cal G}|_Y = (Y,G_\bullet)$ be the subgraph of groups over $Y$ and denote its fundamental pro-$p$ group by $G_Y$. Note that ${\cal G}|_Y$ is proper, as the universal property defining $G_Y$ immediately gives a commuting diagram
\[ \begin{tikzcd}
& G_x \ar{dl} \ar[hook]{dr} &\\
G_Y \ar{rr} && G
\end{tikzcd}\]
for each $x\in Y$. We may define a graph of groups ${\cal H}=(Z, H_\bullet)$ by setting $H_{[x]}=G_x$ for each $x\in X\smallsetminus Y$ and setting $H_{[Y]} = G_Y$. Note that properness of ${\cal G}|_Y$ implies that the natural maps from an edge group $G_e$ of ${\cal G}$ into $G_Y$, for $e$ having an endpoint in $Y$, are indeed injections.
\begin{prop}\label{CollapsingSubgs}
With notation and conditions as above, $\cal H$ is proper and has fundamental group naturally isomorphic to $G$.
\end{prop}
\begin{proof}
The isomorphism $\Pi_1{(\cal H)}\iso \Pi_1(\cal G)$ holds because the two groups may be readily seen to satisfy the same universal property. It remains to show that ${\cal H}$ is proper---i.e. that $G_Y=\Pi_1(Y,G_\bullet)$ injects into $G=\Pi_1(\cal G)$.

We first show that the proposition is valid for graphs ${\cal G}=(X,G_\bullet)$ of finite $p$-groups. In this case we may also consider $\cal G$ as a graph of discrete groups. We have a commuting diagram 
\[\begin{tikzcd}
\pi_1(Y,G_\bullet) \ar{r} \ar{d} &\pi_1({\cal G}) \ar{d}\\
\Pi_1(Y,G_\bullet) \ar{r} &\Pi_1({\cal G}) 
\end{tikzcd}\]
where, by Proposition \ref{ProPvsDisc}, the lower half of the square is the pro-$p$ completion of the upper half.

The top horizontal map is an injection by the theory of graphs of discrete groups---for example, this follows from the expression of elements as reduced words \cite[Section I.5.2]{SerreTrees}. This alone is not enough to show that $\Pi_1(Y,G_\bullet)$ injects into $\Pi_1({\cal G})$: we need further separability properties of $\pi_1(Y,G_\bullet)$ in $\pi_1({\cal G})$. Specifically \cite[Lemma 3.2.6]{RZ00} we must prove that for each normal subgroup $N$ of $\pi_1(Y,G_\bullet)$ of index a power of $p$, there is a normal subgroup $M$ of $\pi_1({\cal G})$ whose intersection with $\pi_1(Y,G_\bullet)$ is contained in $N$. It is enough to prove this property instead for $F\cap \pi_1(Y,G_\bullet)$ and $F$ for $F$ a normal subgroup of $\pi_1(\cal G)$ with index a power of $p$. 

Since $\cal G$ is proper, there is a map $\phi\colon \pi_1({\cal G})\to K$, for $K$ a finite $p$-group, whose restrictions to all $G_x$ are injections. Then $F=\ker(\phi)$ is a free group, of which $F\cap \pi_1(Y,G_\bullet)$ is a free factor and therefore there is a retraction $\rho\colon F\to F\cap \pi_1(Y,G_\bullet)$ (see \cite[Section I.5.5, Theorem 14]{SerreTrees} and the remarks following it). Then for $N$ a normal subgroup of $F\cap \pi_1(Y,G_\bullet)$ with index a power of $p$, the subgroup $M=\rho^{-1}(N)$ has intersection exactly $N$ with $\pi_1(Y,G_\bullet)$ and we are done.

We now deduce the general case. Let $U\nsgp[\rm o] G$ be any open normal subgroup of $G$. We can then form a graph of groups ${\cal G}/U = (X, G_\bullet / G_\bullet\cap U)$ with finite edge groups, which is proper since there is a map to the finite $p$-group $G/U$ restricting to an injection on each $G_x/G_x\cap U$. The universal properties of the various graphs of groups then give a natural commuting diagram
\[\begin{tikzcd}
\Pi_1(Y,G_\bullet)\ar{r}\ar{d}& \Pi_1(X, G_\bullet)\ar{d} \\
\Pi_1(Y,G_\bullet/G_\bullet \cap U)\ar[hook]{r}& \Pi_1(X, G_\bullet/G_\bullet\cap U)
\end{tikzcd}\]
where the injectivity of the lower arrow follows from the first part of the proof. By naturality and Proposition \ref{invlimgofgs}, the top arrow is the inverse limit of the bottom arrows over all $U\nsgp[\rm o] G$. An inverse limits of injections is an injection, thus proving the theorem.
\end{proof}

\section{Uniformly bounded edge groups}\label{secUBEG}
Linnell \cite{Linnell83} proved that finitely generated discrete groups satisfy a form of accessibility provided one imposes a bound on the size of all edge groups rather than simply requiring them to be finite. In this section we prove a similar result for pro-$p$ groups. The proof uses residual finiteness to deduce the theorem from the case of trivial edge groups, and proceeds using the duality between finite graphs of pro-$p$ groups and actions on pro-$p$ trees. This duality is entirely analogous to the classical duality of Bass-Serre theory for discrete groups \cite{SerreTrees}. See \cite[Sections 6.4 and 6.6]{Ribes17} for an account of the relavant theory.
\begin{theorem}\label{thmUBEG}
Let $G$ be a finitely generated pro-$p$ group and let $K$ be an integer. Let $(X, G_\bullet)$ be a proper reduced finite graph of groups decomposition of $G$ such that each edge group $G_e$ has size at most $K$. Then $X$ has at most
\[\frac{pK}{p-1}(\rk(G)-1) + 1 \]
edges.
\end{theorem}
\begin{proof}
Since the graph of groups is proper, we may treat the $G_x$ (for $x\in X$) as (closed) subgroups of $G$. For each $e\in EX$ and each endpoint $x$ of $e$ such that $G_e\neq G_x$, there is a finite quotient of $G$ such that the image of $G_x$ is not equal to the image of $G_e$. Furthermore the $G_e$ are finite. Therefore there exists a finite quotient $\phi\colon G\to P$, with kernel $H$, with the properties that $\phi|_{G_e}$ is an injection for all $e\in EX$ and such that $\phi(G_e)\neq \phi(G_x)$ if $x$ is an endpoint of $e$ for which $G_e\neq G_x$.

The action of $H$ on the $G$-tree $T$ dual to the graph of groups $(X,G_\bullet)$ gives a graph of groups decomposition $(Y,H_\bullet)$ of $H$ where $Y= T/H$. The groups $H_\bullet$ are given by $H$-stabilizers in $T$, so that for $y\in Y$ which maps to $x\in X$ under the quotient map $T/H\to T/G$, the group $H_y$ is a $G$-conjugate of $H\cap G_x$. In particular, since the restriction of $\phi$ to $G_e$ is injective, the graph of groups $(Y,H_\bullet)$ has trivial edge groups. Furthermore $P$ acts on $Y$ with quotient $X$, and with stabilisers which are conjugates of $G_x/H\cap G_x=\phi(G_x)$. Thus $x\in X$ has $|P/\phi(G_x)|$ preimages in $Y$, each with group $H_y$ isomorphic to $H\cap G_x$. 

Factoring out the $H_y$ gives a surjection from $H$ to a free group of rank $1-\chi(Y)$, giving an inequality
\[1-\chi(Y)\leq \rk(H). \]
On the other hand, since $H$ is an index $|P|$ subgroup of $G$, its rank must be bounded above by
\[ \rk(H)\leq (\rk(G)-1)|P| + 1\]
Calculating $\chi(Y)$ using the vertex and edge counts of $Y$ from above and cancelling a common factor of $|P|$ gives the inequality
\[\rk(G) -1 \geq \sum_{e\in EX} \frac{1}{|G_e|} - \sum_{v\in VX} \frac{1}{|\phi(G_v)|}\]
To tame the negative term, choose a maximal rooted subtree $Z$ of $X$ and direct it away from the root. Then each vertex $x$ of $X$ other than the root $r$ is the terminal point $d_1(e)$ of exactly one incoming edge $e$ in $Z$. Since $(X,G_\bullet)$ is reduced, the only edges for which $G_e$ may equal an endpoint are loops, and therefore are not in $Z$. Hence
\[p|G_{e}|\leq |\phi(G_{t(e)})|\]
for each $e\in Z$. Combining these inequalities, we find
\[\rk(G)-1\geq \sum_{e\notin Z} \frac{1}{G_e} + \sum_{e\in Z} \left(1-\frac{1}{p}\right)\frac{1}{|G_e|} - \frac{1}{\phi(G_r)}\]
Finally take some edge $e_0$ adjacent to $r$. If $e_0\in Z$ then once again
\[p|G_{e_0}|\leq |\phi(G_r)|\]
Otherwise, $|G_{e_0}| \leq |\phi(G_r)|$. Either way, one of the $|EX|$ positive terms above is at least as great as the final negative term. We gather the other $|EX|-1$ terms together, recalling that $|G_e|\leq K$ for all $e$, to obtain the final inequality
%\[\rk(G)-1\geq \sum_{e\notin Z\cup e_0} \frac{1}{G_e} + \sum_{e\in Z\smallsetminus e_0} \left(1-\frac{1}{p}\right)\frac{1}{|G_e|} +  \left(1-\frac{1}{p} -1\right)\frac{1}{|G_{e_0}|} \]
%This last term may vanish if $p=2$, but only accounts for one edge so can be safely ignored. Since all the $G_e$ have size at most $K$ we find
\[\rk(G)-1 \geq \left(1-\frac{1}{p}\right)\frac{|EX|-1}{K}\]
as required. 
\end{proof}
\begin{rmk}
From the above calculations one may also find an inequality
\[\frac{p}{p-1}\rk(G) \geq \sum_{e\in EX} \frac{1}{|G_e|}. \]
This should be compared with the inequality 
\[ 2 {\rm d}_{\Q G}(\Q\mathfrak{g})-1 \geq \sum_{e\in EX} \frac{1}{|G_e|}\]
found by Linnell \cite[Theorem 2]{Linnell83}, where $\Q\mathfrak{g} = \ker(\Q G\to \Q)$ is the augmentation ideal.
\end{rmk}

\section{An inaccessible finitely generated pro-$p$ group}\label{secInacc}
\subsection{Introduction}
In this section we will adapt Dunwoody's construction of an inaccessible discrete group \cite{Dunwoody93} to the pro-$p$ category to prove the existence of an inaccessible finitely generated pro-$p$ group. First let us briefly recall the construction from \cite{Dunwoody93}. Take a diagram of groups as follows, where arrows denote proper inclusions.
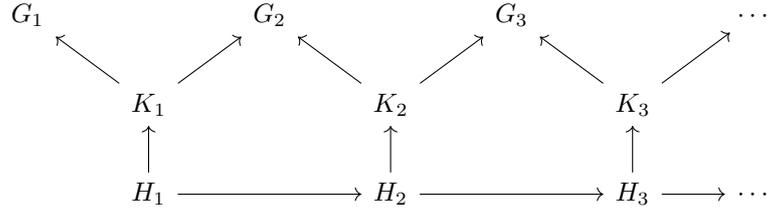
\[ \begin{tikzcd}
G_1 && G_2 && G_3 && \cdots\\
& K_1 \ar{ul}\ar{ur} && K_2 \ar{ul}\ar{ur} && K_3 \ar{ul}\ar{ur} &  \\
& H_1 \ar{u}\ar{rr} && H_2 \ar{u}\ar{rr} && H_3 \ar{u}\ar{r} & \cdots
\end{tikzcd}\]
Assume the following hypotheses:
\begin{itemize}
\item $G_1$ is finitely generated;
\item for all $i$, the group $G_{i+1}$ is generated by $K_i$ and $H_{i+1}$; and
\item each $K_i$ is finite.
\end{itemize}
Then the group $P$ defined as the fundamental group of the infinite graph of groups
\begin{equation}\label{infgofgs}
\begin{tikzcd}
 G_1 \ar[dash]{r}{K_1} & G_2 \ar[dash]{r}{K_2} & G_3 \ar[dash]{r}{K_3}& \cdots 
\end{tikzcd}\end{equation}
has arbitrarily large splittings over finite groups. It may however not be finitely generated. This problem is overcome by noting that the countable group $H_\infty = \bigcup H_n$ embeds in some finitely generated group $H_\omega$. Then by construction $J=P\ast_{H_\infty} H_\omega$ is generated by $G_1$ and $H_\omega$, and admits splittings
\[\begin{tikzcd}
 G_1 \ar[dash]{r}{K_1} & \cdots \ar[dash]{r}{K_{n-1}}& G_n \ar[dash]{r}{K_n} &\big( G_{n+1} \ar[dash]{r}{K_{n+2}}& G_{n+2} \ar[dash]{r}{K_{n+3}} & \cdots \big) \ar[dash]{r}{H_\infty} & H_\omega
\end{tikzcd}\]
for all $n$. In applying this scheme to pro-$p$ groups, there are several additional considerations to worry about.
\begin{enumerate}
\item Each graph of groups involved must be confirmed to be proper. If improper `splittings' are allowed, the theory quickly becomes absurd.
\begin{example}[Arbitrarily large improper graphs of finite $p$-groups with fundamental group $\F_p\amalg \F_p$]\label{ArbLargeImprop}
Let $G_n$ be a copy of the mod-$p$ Heisenberg group
\[ G_n = \big<a_n, b_n\,\big|\, a_n^p=b^p_n=[a_n, b_n]^p=1, [a_n, b_n]\text{ central} \big>\]
and let $K_n$ be a copy of $\F_p^2$ generated by $u_n$ and $v_n$. Take inclusions
\[\begin{matrix} K_n\hookrightarrow G_{n-1}, & u_n\mapsto b_{n-1},& v_n\mapsto [a_{n-1}, b_{n-1}] \\K_n\hookrightarrow G_n, & u_n\mapsto [a_{n}, b_{n}],& v_n\mapsto a_n \end{matrix}\]
and form the graph of groups 
\[\begin{tikzcd}
 G_1 \ar[dash]{r}{K_1} & G_2 \ar[dash]{r}{K_2} & G_3 \ar[dash]{r}{K_3}& \cdots\ar[dash]{r}{K_{N-1}} &G_N
\end{tikzcd} \]
for $N>1$. In the fundamental group $G$ of this graph of groups we have identities
\[b_{n-1}=[a_{n}, b_{n}],\quad [a_{n-1}, b_{n-1}] = a_n \]
for $2\leq n\leq N$. Iterating these relations we find that $a_n$ and $b_{n-1}$ may be expressed as commutators of arbitrary length; in the pro-nilpotent group $G$ this forces them to be trivial. Thus $G$ is generated only by the order $p$ elements $a_1$ and $b_N$ and is readily seen to be $\F_p\amalg \F_p$.
\end{example}
\item Not every increasing union of finite $p$-groups embeds in a finitely generated pro-$p$ group. Furthermore, care must be taken to embed not only the union of the $H_i$ in a finitely generated pro-$p$ group, but also their closure in $P$. Two embeddings of such a union in different pro-$p$ groups could well have non-isomorphic closures.
\begin{example}[An increasing union of finite $p$-groups not embeddable in any pro-$p$ group]
Consider $\Gamma = \{z\in \mathbb{C} \mid z^{p^n}=1\text{ for some }n\}$, which is a union of finite cyclic $p$-groups. We claim that any map from $\Gamma$ to a pro-$p$ group is trivial. It suffices to check this for finite $p$-groups. Let $\phi\colon\Gamma\to P$ be a map to a finite $p$-group. Each $g\in \Gamma$ has a $|P|^{\rm th}$ root $h$, so that $\phi(g)=\phi(h)^{|P|}=1$.
\end{example}
\begin{example}[Distinct pro-$p$ groups generated by the same union of finite $p$-groups]
Consider the following two abelian pro-$p$ groups. Let $H'=\pRm[\F_p]{\Z[p]}$ be the free $\F_p$ module on the profinite space $\Z[p]$. Let $\widetilde \Z$ be Alexandroff compactification of $\Z$---the set $\Z$ compactified at a single point $\ast$---and let $H= \pRm[\F_p]{(\widetilde\Z, \ast)}$ be the free $\F_p$-module on the pointed profinite space $(\widetilde\Z, \ast)$. See \cite[Chapter 5]{RZ00} for information on free profinite modules. By examining the standard inverse limit representations
\[\pRm[\F_p]{(\widetilde\Z, \ast)}=\varprojlim_k \F_p[(\{-k\ldots, k,\ast\}, \ast)], \qquad\pRm[\F_p]{\Z[p]} = \varprojlim_n \F_p[\Z/p^n] \]
of these two groups one verifies that they are generated by the union of the finite subgroups $\F_p[\{-k,\ldots,k\}]$. 

However there is no isomorphism from $H$ to $H'$ preserving this family of subgroups. Since $0$ is an isolated point of $(\widetilde\Z, \ast)$ there is a projection $H\to \F_p[\{0\}]$ killing the other factors $\F_p[\{k\}]$ for $k\neq 0$. However in the indexing set for $H'$ the point $0$ is in the closure of $\Z\smallsetminus \{0\}$, hence any map killing all the factors $\F_p[\{k\}]$ for $k\neq 0$ must also kill $\F_p[\{0\}]$.

\end{example}
\item In dealing with graphs of profinite groups, not only the groups but also the {\em graphs} involved must be profinite to retain a sensible topology. Hence forming a graph of groups such as \eqref{infgofgs} is an invalid operation. Our construction will use an inverse limit of finite graphs of groups. This can also be rephrased as a profinite graph of pro-$p$ groups which is a certain `compactification' of the graph of groups \eqref{infgofgs}, though we will not make this explicit. 
\end{enumerate} 
\subsection{Definition of groups}
Now we will describe the promised inaccessible pro-$p$ group. We will construct a diagram of pro-$p$ groups as in Dunwoody's construction, together with retraction maps which will allow us to define an inverse limit. Assume for this discussion that all graphs of groups given are proper---we will verify this in the next subsection.

First define the map \[\mu_n\colon \{0, \ldots, p^{n+1}-1 \}\to \{0, \ldots, p^{n}-1 \}\]
by sending an integer to its remainder modulo $p^n$. Define
\[H_n = \F_p[\{0, \ldots, p^{n}-1 \}]  \]
to be the $\F_p$-vector space with basis $\{h_0,\ldots, h_{p^n-1}\}$. There are inclusions $H_n\subseteq H_{n+1}$ given by inclusions of bases, and retractions $\eta_n\colon H_{n+1}\to H_n$ defined by $h_k\mapsto h_{\mu_n(k)}$. 

Note also that there is a natural action of $\Z/p^n$ on $H_n$ given by cyclic permutation of the basis elements, and that these actions are compatible with the retractions $\eta_i$. The inverse limit of the $H_n$ along these retractions is $H_\infty = \pRm[\F_p]{\Z[p]}$. Since the positive integers are dense in \Z[p], the group $H_\infty$ is generated by the images of the $H_n$ under inclusion. The continuous action of $\Z[p]$ on the given basis of $H_\infty$ allows us to form a sort of `pro-$p$ lamplighter group'
\[ H_\omega = \pRm[\F_p]{\Z[p]}\rtimes \Z[p] = \varprojlim (H_n\rtimes \Z/p^n) \]
which is a pro-$p$ group into which $H_\infty$ embeds. As will be seen later, $H_\omega$ is finitely generated.

Next set \[K_n = \F_p\times H_n = \gp{k_n}\times H_n.\] Set $G_1= K_1 \times \F_p$. For $n>1$, let $G_n$ be the pro-$p$ group with presentation
\[\begin{split}
G_n=\big<k_{n-1}, k_n,& h_0, \ldots, h_{p^n-1} \,|\, k_i^p = h_i^p =1,h_i\leftrightarrow h_j,\\ & k_{n-1} \leftrightarrow h_i \text{ for all } i\neq p^{n-1},  k_n = [k_{n-1}, h_{p^{n-1}}] \text{ central}\big>
\end{split} \]
where $\leftrightarrow$ denotes the relation `commutes with'. One may show that in fact $G_n$ is isomorphic to the finite $p$-group whose underlying set is
\begin{equation*} \label{GnExpl}
G_n = \{(u_{n-1}, u_n, x_0, \ldots, x_{p^n-1}) \in \F_p^{2+p^n}\} 
\end{equation*}
with operation 
\[\begin{split} (u_{n-1}, u_n, x_0,& \ldots, x_{p^n-1})\star (v_{n-1}, v_n, y_0, \ldots, y_{p^n-1}) =\\ &(u_{n-1}+v_{n-1}, u_n + v_n + u_{n-1}y_{p^{n-1}}, x_0+y_0, \ldots, x_{p^n-1}+y_{p^n-1}) \end{split}\] 
where the $u_i$-coordinate represents $k_i$ and the $x_i$-coordinate represents $h_i$.

 The choice of generator names describes maps $H_n\to G_n$, $K_{n-1}\to G_n$, and $K_n\to G_n$. Using the explicit description of $G_n$ one may easily see that all three of these maps are injections. Define a retraction map $\rho_n\colon G_n\to K_{n-1}$ by killing $k_n$ and by sending $h_k\mapsto h_{\mu_{n-1}(k)}$. Note that $\rho_n$ is compatible with $\eta_n\colon H_n\to H_{n-1}$---that is, there is a commuting diagram 
\[\begin{tikzcd}
K_{n-1} \ar[hook,shift left =1]{r} & \ar[shift left = 1]{l}{\rho_n} G_n\\
H_{n-1} \ar[hook]{u} \ar[hook,shift left =1]{r} & \ar[shift left = 1]{l}{\eta_n} H_n \ar[hook]{u}
\end{tikzcd}\]

Let $Q_{n,m}$ be the fundamental pro-$p$ group of the graph of groups
\begin{equation}\label{Pdefn}\begin{tikzcd}
Q_{n,m}\colon &  G_{n+1} \ar[dash]{r}{K_{n+1}} & G_{n+2} \ar[dash]{r}{K_{n+2}} & \cdots \ar[dash]{r}{K_{n+m-1}} & G_{n+m}
\end{tikzcd}  
\end{equation} 
and let $P_m = Q_{0,m}$. By Proposition \ref{CollapsingSubgs} the group $Q_{n,m}$ includes naturally in $Q_{n,m+1}$. Note that the retractions $\rho_n\colon G_n\to K_{n-1}$ induce retractions $\pi_{n,m}\colon Q_{n,m+1}\to Q_{n,m}$ for all $n$ and $m$:
\[\begin{tikzcd}
Q_{n,m+1}\colon \ar{d}{\pi_{n,m}}&  G_{n+1} \ar[dash]{r}{K_{n+1}} \ar{d}{\id}& \cdots \ar[dash]{r}{K_{n+m-1}} & G_{n+m} \ar[dash]{r}{K_{n+m}} \ar{d}{\id}& G_{n+m+1}\ar{d}{\rho_{n+m+1}} \\
Q_{n,m}\colon \ar[shift left =2, hook]{u}&  G_{n+1} \ar[dash]{r}{K_{n+1}} \ar[shift left =2, hook]{u}& \cdots \ar[dash]{r}{K_{n+m-1}} & G_{n+m}\ar[dash]{r}{K_{n+m}} \ar[shift left =2, hook]{u}& K_{n+m}\ar[shift left =2, hook]{u}
\end{tikzcd}\]
Define $Q_{n,\infty}$ to be the inverse limit of the $Q_{n,m}$ along these retractions, and set $P = Q_{0,\infty}$. 

By Proposition \ref{CollapsingSubgs} we have an identity $P_n\amalg_{K_n} Q_{n,m} = P_{n+m}$. Taking an inverse limit with Proposition \ref{invlimgofgs} we find that $P$ has a graph of groups decomposition
\begin{equation}\label{Pinacc}
 \begin{tikzcd}
P\colon &  G_{1} \ar[dash]{r}{K_{1}} & G_{2} \ar[dash]{r}{K_{2}} & \cdots \ar[dash]{r}{K_{m-1}} & G_{m} \ar[dash]{r}{K_{m}} &Q_{n,\infty}
\end{tikzcd}
\end{equation} 
for all $n$; so $P$ is inaccessible. 

To complete the last part of Dunwoody's construction we are required to take an amalgamated product with the finitely generated group $H_\omega$ to restore finite generation. Because the $H_n$ generate $H_\infty$ and there is an injection $H_\infty \to P$ defined to be the inverse limit of the injections $H_n\to P_n$, it follows that the closure of the union of the $H_n$ in $P$ is isomorphic to $H_\infty$. So consider the amalgamated free product $J=P\amalg_{H_\infty} H_\omega$. Both $P$ and $H_\omega$ are defined as inverse limits whose transition maps both restrict to the maps $\eta_n$ on $H_n$. Hence we may write
\[J = P \amalg_{H_\infty} H_\omega = \varprojlim_n P_n \amalg_{H_n} (H_n\rtimes \Z/p^n).\]
The graphs of groups attesting that $J$ is inaccessible are the splittings
\begin{equation}\label{Jinacc}
\begin{tikzcd}
J\colon &  G_{1} \ar[dash]{r}{K_{1}} & G_{2} \ar[dash]{r}{K_{2}} & \cdots \ar[dash]{r}{K_{n-1}} & G_{n} \ar[dash]{r}{K_{n}} &Q_{n,\infty} \ar[dash]{r}{H_\infty} & H_\omega\end{tikzcd}
\end{equation}
for each $n$. By Proposition \ref{invlimgofgs} these are the inverse limits of the graphs of groups
\begin{equation}\label{Jinvlt}
\begin{tikzcd}
G_{1} \ar[dash]{r}{K_{1}} & G_{2} \ar[dash]{r}{K_{2}} & \cdots \ar[dash]{r}{K_{n-1}} & G_{n} \ar[dash]{r}{K_{n}} &Q_{n,m} \ar[dash]{r}{H_{n+m}} & H_{n+m}\rtimes \Z/p^{n+m}.
\end{tikzcd} 
\end{equation}

\subsection{Properness}\label{subsecProper}
We must now show that the various graphs of groups introduced in the previous section are proper.

First consider the graph of groups \eqref{Pdefn}. It suffices to deal with the case $n=0$ since a subgraph of a proper graph of groups is proper. We exhibit an explicit finite quotient of $P_n$ into which all the $G_i$ inject for $1\leq i\leq n$, thereby establishing properness. The group presentation 
\[\begin{split}
F_n = \big<k_1, \ldots, k_{n-1}, k_n, h_0, \ldots, h_{p^n-1} \,|\,& k_i^p = h_i^p =1, k_i \leftrightarrow k_j, h_i\leftrightarrow h_j,\\ & k_i \leftrightarrow h_j \text{ for all } j\neq p^{i},  k_{i+1} = [k_i, h_{p^{i}}]\big>
\end{split} \]
defines a finite $p$-group with an explicit form similar to the description of $G_n$ given above: specifically the set 
\[F_n = \{((u_1,\ldots, u_n), (x_0, \ldots, x_{p^n-1}))\in \F_p^n \times \F_p^{p^n}\} \]
with multiplication 
\[(\underline u, \underline{x})\star (\underline{v}, \underline{y}) = ((u_i + v_i + u_{i-1}y_{p^{i-1}}), \underline{x}+ \underline{y}). \]
The choice of generator names specifies natural maps from the $G_i$ and $K_i$ into $F_n$, compatible with the inclusions of $K_i$ into $G_i$ and $G_{i+1}$. This therefore describes a map $P_n\to F_n$, and via the explicit description of $F_n$ one readily sees that the $G_i$ inject into $F_n$ under this map. This shows that the graph of groups \eqref{Pdefn} is proper. An application of Proposition \ref{invlimgofgs} shows that the inverse limit \eqref{Pinacc} of these is also proper.

To show that the graphs of groups \eqref{Jinvlt} are proper, we require a finite quotient of $P_{n} \amalg_{H_n} H_n\rtimes (\Z/p^n)$ into which all the $G_i$ and $H_n\rtimes \Z/p^n$ inject. Such a finite group $E_n$ is given by the presentation below; one could provide an explicit finite $p$-group isomorphic to it (similar to the treatments of $G_n$ and $F_n$ above), but we leave this to the reader with a penchant for such activities. The group presentation is
\begin{multline*}
E_n = \big< k_{i,r} \,(1\leq i\leq n, 0\leq r<p^n), h_0, \ldots, h_{p^{n}-1}, t \,\mid\, k_{i,r}^p= h_i^p= t^{p^n}=1,\\
k_{i,r}\leftrightarrow k_{j,s},\quad h_i\leftrightarrow h_j,\quad k_{i,r}\leftrightarrow h_j \text{ if $j\neq p^i + r$},\\ k_{i+1, r} = [k_{i,r}, h_{p^i+r}],\quad
t^{-1} k_{i,r} t = k_{i, r+1},\quad t^{-1}h_jt = h_{j+1} \big>
\end{multline*}
where the indexing set of the $h_i$ should be taken modulo $p^n$, as should the $r$-coordinate of the indexing set of the $k_{i,r}$. By Proposition \ref{invlimgofgs} we now also know that the graph of groups \eqref{Jinacc} is proper for each $n$.
\subsection{Conclusion}
By construction, $J$ is expressed as the fundamental group of arbitrarily large reduced graphs of groups 
\eqref{Jinacc}, which we now know to be proper---that is, $J$ is an inaccessible group. It only remains to argue that $J$ is finitely generated. 

As per Dunwoody's construction, $J$ is generated by $G_1$ and $H_\omega$. By definition $J$ is generated by $H_\omega$ and the $G_n$, and each $G_n$ is generated by $K_{n-1} \subseteq G_{n-1}$ and $H_n$; so by induction the subgroup of $J$ generated by $G_1$ and $H_\omega$ contains all the $G_n$. 

The finite group $G_1$ is of course finitely generated. The `pro-$p$ lamplighter group' $H_\omega$ is also finitely generated---it is generated by the two elements 
\[ (h_0 , 0), (0, t) \in \pRm[\F_p]{\Z[p]}\rtimes \Z[p] \] 
where $t$ denotes a generator of the second factor $\Z[p]$ in the semidirect product. That these elements do indeed generate the group follows from the fact that their images generate each term in the inverse limit 
\[H_\omega=\varprojlim (H_n\rtimes \Z/p^n)\]
defining $H_\omega$. 

Since both $G_1$ and $H_\omega$ are finitely generated, so is $J$. We have now proved our primary goal, the following theorem.
\begin{theorem}\label{thmInacc}
There exists a finitely generated inaccessible pro-$p$ group $J$.
\end{theorem}

\section{A residually $p$-finite inaccessible group}\label{secRespInacc}
The focus of this paper has been pro-$p$ groups. However the constructions made above also allow us to give, to the best of the author's knowledge, the first example in the literature of a finitely generated inaccessible {\em discrete} group which also has the property that it is residually $p$-finite (or even residually finite). We now describe this, with a brief proof.

Let the groups $G_i$, $K_i$, $H_i$, and $E_i$ be as in the previous section. Define $P_n$ to be the (discrete) fundamental group of the graph of groups
\[\begin{tikzcd}
P_n\colon &  G_{1} \ar[dash]{r}{K_{1}} & G_{2} \ar[dash]{r}{K_{2}} & \cdots \ar[dash]{r}{K_{n-1}} & G_{n} 
\end{tikzcd}\]
and let $P$ be the (discrete) fundamental group of the infinite groups 
\begin{equation*}
\begin{tikzcd}
P\colon &  G_{1} \ar[dash]{r}{K_{1}} & G_{2} \ar[dash]{r}{K_{2}} & \cdots \ar[dash]{r}{K_{m-1}} & G_{m} \ar[dash]{r}{K_{n}} &\cdots
\end{tikzcd}
\end{equation*} 
Let $H_\infty = \F_p[\N]=\bigcup_{n} H_n$, and let $H_\omega= \F_p[\Z{}]\rtimes \Z$.  As per Dunwoody's original construction, the amalgamated free product $J=P\ast_{H_\infty} H_\omega$ is an inaccessible finitely generated discrete group. 

To show that $J$ is residually $p$-finite, let $j\in J\smallsetminus\{1\}$. The element $J$ is given by a reduced word---that is, an expression
\[ j= a_1b_1\cdots a_mb_m\]
where 
\begin{itemize}
\item $a_i\in P$ for all $i$ and $b_i\in H_\omega$ for all $i$;
\item no $a_i$ or $b_i$ is trivial, except possibly $a_1$ or $b_m$; and
\item no non-trivial $a_i$ or $b_i$ is contained in $H_\infty$ except possible $a_1$.
\end{itemize}
Now $P$ is the union of the $P_n$, so for all $n$ sufficiently large all $a_i$ are contained in $P_n$---and therefore are fixed by the retraction $P\to P_n$. For all $n$ sufficiently large, the images of the $b_i$ under the quotient map $H_\omega \to H_n\rtimes \Z/p^n$ do not lie in $H_n$. Therefore for all $n$ sufficiently large, the image $j'$ of $j$ in $P_n \ast_{H_n} (H_n\rtimes \Z/p^n)$ is given by a reduced word, hence is non-trivial. This latter group is residually $p$-finite; this follows from Section \ref{subsecProper} since the kernel of the map $P_n \ast_{H_n} (H_n\rtimes \Z/p^n) \to E_n$ is free, hence is residually $p$-finite. Therefore there is a finite $p$-group quotient $\phi\colon P_n \ast_{H_n} (H_n\rtimes \Z/p^n) \to L$ such that $\phi(j')\neq 1$. Thus the composition
\[J \to P_n \ast_{H_n} (H_n\rtimes \Z/p^n) \to L \]
does not kill $j$. It follows that $J$ is residually $p$-finite as required.
\begin{theorem}\label{ThmRespInacc}
There exists a finitely generated, inaccessible, residually $p$-finite discrete group.
\end{theorem}

\subsection*{Acknowledgements}
The author was supported by a Junior Research Fellowship from Clare College, Cambridge.

\bibliographystyle{alpha}
\bibliography{AccMain.bib}
\end{document}